\documentclass[12pt]{article}

\usepackage{amssymb}
\usepackage{amsmath}
\usepackage{amsbsy}
\usepackage{amscd}
\usepackage{amsfonts}
\usepackage{amsthm}
\usepackage{mathrsfs}
\usepackage{verbatim}
\usepackage[colorlinks]{hyperref}
\usepackage{fullpage}
\usepackage{mathdots}
\usepackage{graphicx,subfigure}
\usepackage[english]{babel}
\usepackage[utf8]{inputenc}
\usepackage{tikz}
\usepackage{adjustbox}
\usepackage{authblk}
\usepackage{caption}

\usepackage{mathtext}

\usetikzlibrary{backgrounds,fit, matrix}
\usetikzlibrary{positioning}
\usetikzlibrary{calc,through,chains}
\usetikzlibrary{arrows,shapes,snakes,automata, petri}

\newtheorem{Theorem}[equation]{Theorem}
\newtheorem{Corollary}[equation]{Corollary}
\newtheorem{Lemma}[equation]{Lemma}
\newtheorem{Proposition}[equation]{Proposition}

\theoremstyle{definition}
\newtheorem{Definition}[equation]{Definition}
\newtheorem{Example}[equation]{Example}

\newtheorem{Remark}[equation]{Remark}

\numberwithin{equation}{section}
\numberwithin{figure}{section}

\newcommand{\C}{{\mathbb C}}

\newcommand{\mc}[1]{\mathcal{#1}}

\newcommand{\beq}{\begin{equation}}
\newcommand{\eeq}{\end{equation}}

\begin{document}

\title{Complexity $c$ Pairs in Simple Algebraic Groups}

\author{Mahir Bilen Can}
\affil[1]{{mahirbilencan@gmail.com}}

\maketitle

\begin{abstract}
We call a pair of closed subgroups $(G_1,G_2)$ from a connected reductive algebraic group $G$ 
a {\it complexity $c$ pair} if the multiplication action of the pair on $G$ is of complexity $c$. 
The main focus of this article is on the cases where $G$ is simple and $c$ is either 0 or 1. 
After showing that both of the subgroups $G_1$ and $G_2$ cannot be reductive subgroups unless
$c>1$, we look for the cases where exactly one of the subgroups $G_1$ and $G_2$ is reductive. 
It turns out that there are only a few such pairs, and their classification involves the horospherical homogeneous spaces of small ranks. 
As a byproduct of the circle of ideas that we use for this development, we obtain the classification of the diagonal spherical 
actions of simple algebraic groups on the products of flag varieties with affine homogeneous spaces. 
\vspace{.5cm}


\end{abstract}

\normalsize

\section{Introduction}

Let $G$ be a connected reductive algebraic group, and let $B$ be a Borel subgroup. 
Let $X$ be an irreducible normal variety on which $G$ acts morphically. 
The {\em complexity} of the action $G:X$, denoted by $c_G(X)$, is the minimal codimension of a $B$-orbit in $X$, that is,
\hbox{$c_G(X) := \min_{x\in X} \text{codim}_X Bx$}.
If $c_G(X)=0$, then $X$ is called a {\em spherical $G$-variety}, 
and the action $G:X$ is called a {\em spherical action}. 
Let $(G_1,G_2)$ be a pair of closed subgroups from $G$, and 
let us call the morphism \hbox{$(G_1\times G_2) \times G \to  G$} 
defined by \hbox{$((g_1,g_2),h) \mapsto g_1 h g_2^{-1}$} the {\em natural action of $(G_1,G_2)$ on $G$}.

\begin{Definition}
Let $c$ be a nonnegative integer. The pair $(G_1,G_2)$ is called a {\em complexity $c$ pair in $G$}
if the complexity of the natural action of $(G_1,G_2)$ on $G$ is equal to $c$. 
In particular, $(G_1,G_2)$ is called a {\em spherical pair in $G$} if it is a complexity zero pair in $G$. 
\end{Definition}

Note that since Borel subgroups are always connected, we assume without further mentioning in the sequel that 
if $(G_1,G_2)$ is any pair of subgroups from $G$, then both of the subgroups $G_1$ and $G_2$ are connected. 
This assumption does not cause any loss of generality for our purposes. 

In the paper, we have two main results regarding the complexity of actions for a pair of subgroups. 
Our first result is about the complexity zero and complexity one pairs in simple algebraic groups. 
In particular, we look closely on the situation where at least one of the subgroups $G_1$ and $G_2$ is reductive. 
As we will show in the sequel, there is no complexity zero, or complexity one pair $(G_1,G_2)$ such that both of 
the subgroups $G_1$ and $G_2$ are reductive. 
One of our main results is about the classification of pairs $(G_1,G_2)$, where only the first subgroup is reductive, 
and the pair $(G_1,G_2)$ is of complexity one. More precisely, we have the following result. 
\vspace{.5cm}

\textbf{Theorem A.}
Let $G$ be a simple algebraic group, and let $(G_1,G_2)$ be a complexity one pair in $G$. 
If $G_1$ is a reductive subgroup, then 
$$
(G,G_1)\in \{ (\text{SO}(4),\text{GL}(2)), (\text{SO}(4),\text{SO}(3)), (\text{SO}(3),\text{SO}(2)),(\text{SL}(2),\{e\})\}.
$$
Furthermore, in all of these cases, $G_2$ contains a maximal unipotent subgroup of $G$. 
\vspace{.5cm}

The second main result of our article is about the diagonal actions of simple groups on the products of homogeneous spaces. 
This is a fruitful subject with important consequences in representation theory, see~\cite[Section 11.4]{Timashev}. 
Here, we will focus on the situation where $G$ acts diagonally on $G/P\times G/H$, where $P$ is  a parabolic subgroup, and $H$ 
is a reductive subgroup. 
In particular, a subgroup $H$ in $G$ is called a {\em symmetric subgroup} if there exists an involutory automorphism $\theta: G\to G$ 
such that $H=\{g\in G:\ \theta(g) =g \}$; all symmetric subgroups are reductive subgroups. 
The second main result of our article is the following. 
\\

\textbf{Theorem B.}
Let $G$ be a simple algebraic group. Up to conjugation and reordering,
there are only four spherical diagonal actions $G: G/P\times G/H$,
where $P$ is a parabolic subgroup and $H$ is a reductive subgroup of $G$.
In the three of these four cases, $H$ is a symmetric subgroup. 
\vspace{.5cm}

We now give a brief outline of our paper. 
In Section~\ref{S:Preliminaries}, we review some well known facts about spherical actions 
and homogeneous bundles. The purpose of Section~\ref{S:ComplexityZero}
is to show that there are no reductive pairs of complexity zero or complexity one. 
In Section~\ref{S:ComplexityOne}, we have some basic observations regarding complexity one pairs. 
The proofs of Theorems A and B are written in 
Sections~\ref{S:Half-reductive} and~\ref{S:Diagonal}, 
as Theorems~\ref{T:FirstClassification} and~\ref{T:SecondClassification}, 
respectively. The pairs of subgroups that appear in these theorems are explicitly determined.

\vspace{.5cm}

\textbf{Acknowledgements.}
We thank Roman Avdeev, Michel Brion, Bill Graham, Aloysius Helminck, 
Maarten van Pruijssen, and John Stembridge.

\section{Preliminaries}\label{S:Preliminaries}

There are several equivalent characterizations of spherical actions.
\begin{Theorem}\label{T:intro}
Let $G$ be a connected reductive group, let $B$ be a Borel subgroup of $G$, 
and let $X$ be an irreducible normal $G$-variety.
Then the following statements are equivalent:
\begin{enumerate}
\item[(1)] $X$ is a spherical $G$-variety;
\item[(2)] the number of $B$-orbits in $X$ is finite; 
\item[(3)] if $X$ is quasi-affine, then the coordinate ring $k[X]$ is a 
multiplicity-free $G$-module. 
\end{enumerate}
\end{Theorem}

The equivalence of (1) and (2) is proven by Brion~\cite{Brion86}, and by Vinberg in~\cite{Vinberg86}.
The equivalence of (1) and (3) is due to Vinberg and Kimelfeld~\cite{VK78}.

\begin{Remark}
Let $X$ be a spherical $G$-variety. 
Clearly, since there are only finitely many $B$-orbits, 
$X$ has only finitely many $G$-orbits.
Less obvious is the fact that each $G$-orbit closure in $X$ 
is a spherical $G$-variety also, see~\cite{LunaVust}.
\end{Remark}

A proof of the following fact can be found in~\cite{Knop95}.
\begin{Lemma}\label{L:Knop}
Let $G$ be a connected reductive group, and let $X$ be a $G$-variety. 
If $Y$ is a $B$-stable irreducible subvariety in $X$, then $c_G(Y)\leq c_G(X)$. 
\end{Lemma}

Next, by following Avdeev and Pethukov~\cite{AP}, we will review 
some well known facts regarding homogeneous bundles.

Let $G$ be a connected algebraic group,
let $G/H$ be a homogeneous space, and let $X$ be $G$-variety
with a surjective $G$-equivariant morphism $\varphi : X \to G/H$. 
Let us denote the fibre of $\varphi$ over the origin $o=eH$ by $Y$.
 We know that the map $\iota : O\mapsto O\cap Y$ is a bijection
 between $G$-orbits in $X$ and $H$-orbits in $Y$, and that 
 $\dim O - \dim (O\cap Y) = \dim X - \dim Y = \dim G/H$, 
 see~\cite[Proposition 4.2]{AP}.
The proof of the following statement follows from definitions. 
 \begin{Corollary}\label{C:veryuseful}
 We preserve the notation from the previous paragraph.
Then there is an open $G$-orbit in $X$ if and only if 
 there is an open $H$-orbit in $Y$. 
 \end{Corollary}
We now assume that $H$ is a closed subgroup 
and $P$ is a parabolic subgroup in $G$.  
We put $X:= G/P\times G/H$. Then $G$ acts on $X$ by
the diagonal action, $(g,(aP,bH)) \mapsto (gaP,gbH)$. 
Let $B$ be a Borel subgroup of $G$ such that its opposite, $B^-$,
is contained in $P$. Let $K$ be a Levi subgroup of $P$.
The intersection $B\cap K$, which we denote by $B_K$, 
is a Borel subgroup of $K$. 
Interpreted geometrically, $B_K$ is the stabilizer subgroup in $B$ of the origin $eP$ 
for the left translation action of $B$ on $G/P$. 
Thus, the orbit at the origin, $O=Bo \cong B/B_K$, is open in $G/P$,
and therefore, $O\times G/H$ is open in $G/P\times G/H$. 
Note that, the first projection, $pr_1 : O\times G/H \to O$,
is $B$-equivariant, and furthermore, it 
gives a homogeneous bundle structure on $O\times G/H$.
Clearly, the fibre over any point $xB_K \in B/B_K$ ($x\in B$) 
is canonically isomorphic to $G/H$. 
It follows from Corollary~\ref{C:veryuseful} 
that $O\times G/H$ has an open $B$-orbit if and only if 
$G/H$ has an open $B_K$-orbit. 
At the same time, by the openness of $O\times G/H$ in $G/P\times G/H$, 
this is equivalent to the statement that $G/P\times G/H$ is a spherical $G$-variety. 
We summarize these observations as a lemma, a version of which is 
first recorded as Lemma 5.4 in~\cite{AP}. 
Note that in their statement, \cite[Lemma 5.4]{AP}, 
Avdeev and Pethukhov assume that $H$ is parabolic, but as 
the above argument shows this assumption is not necessary. 
\begin{Lemma}\label{L:itsolvesit}
We preserve the notation from the previous paragraph.
Then, the following conditions are equivalent:
\begin{enumerate}
\item $G/H$ is a spherical $K$-variety;
\item $G/P\times G/H$ is a spherical $G$-variety.
\end{enumerate} 
\end{Lemma}

We close this section by introducing some useful terminology.  
Let $G_1$ and $G_2$ be two subgroups from $G$. The subset 
\hbox{$G_1G_2:= \{g_1g_2 :\ g_1\in G_1,\ g_2\in G_2\}$}
is called a {\em decomposition of $G$} if $G_1G_2 = G$. 
More generally, a triplet $(G,G_1,G_2)$, where $G$ is an algebraic group,
and $G_1,G_2$ are closed subgroups in $G$, is called a {\em $d$-decomposition} 
if $d$ is the minimal codimension of $G_1$-orbits in $G/G_2$. 
The 0-decompositions of (compact) Lie groups are described by Onishchik in~\cite{Onishchik1,Onishchik2}.
Panyushev used certain $1$-decompositions for classifying reductive subgroups of complexity one in simple groups, see~\cite{Panyushev}.

We will denote by $\mathbf{G}_m$ and $\mathbf{G}_a$ the one dimensional multiplicative group 
$(\C^*,\cdot)$ and the one dimensional additive group $(\C,+)$, respectively.

\section{Observations About Complexity Zero Pairs}\label{S:ComplexityZero}

Recall that a pair of closed subgroups, $(G_1,G_2)$, from $G$ is called a spherical pair if 
$G_1\times G_2:G$ is a spherical action. 
Let us show that a spherical pair is in fact a pair of spherical subgroups.

\begin{Proposition}\label{P:2}
Let $G$ be a connected reductive algebraic group.  
If $(G_1,G_2)$ is a spherical pair in $G$, then $G/G_i$, for $i\in \{1,2\}$, is a spherical $G$-variety. 
\end{Proposition}

\begin{proof}
Let $X$ denote the wonderful compactification of $G/Z(G)$ as a $G\times G$-variety, where $Z(G)$ is the center of $G$. 
It is well known that $X$ is comprised of $2^r$ $G\times G$-orbits, and the open orbit is isomorphic to $G/Z(G)$.
Here, $r$ is the semisimple rank of $G$. 
For $I\subseteq [r]:=\{1,\dots, r\}$, let $X_I$ denote the closure of a $G\times G$-orbit, where $X_\emptyset$ stands for the open orbit, 
and $X_{[r]}$ stands for the unique closed orbit in $X$.

Since the action of $G_1\times G_2$ on $G$ is spherical, and $G$ is open in $X$, 
we see that $X$ is a spherical $G_1\times G_2$-variety. 
In particular, since each orbit closure $X_I$ is $G\times G$-stable, therefore $G_1\times G_2$-stable, 
the finiteness of Borel orbits implies that each $X_I$ ($I\subseteq [r]$) in $X$ is a spherical $G_1\times G_2$-variety.
Note that the closed $G\times G$-orbit in $X$ is isomorphic to the double-flag variety,
\hbox{$X_{[r]} \cong G/B \times G/B^-$}, 
where the action of $G\times G$ is given by 
\begin{align}\label{A:GGaction}
(g,h)\cdot (aB/B,bB^-/B^-) = (gaB, hbB^-/B^-)\ \text{ for $(g,h),(a,b)\in G\times G$. }
\end{align}
Thus, we see that there exists a Borel subgroup $B_1\times B_2$ in $G_1\times G_2$ 
with an open orbit in $X_{[r]}$ with respect to the action (\ref{A:GGaction}).
This means that the action of $G_1$ (respectively, of $G_2$) on $G/B$ (respectively, on $G/B^-$) is spherical. 
In other words, $G_1$ and $G_2$ are spherical subgroups of $G$.
\end{proof}

\begin{Example}\label{E:goodexample}
The purpose of this example is to show that the converse of Proposition~\ref{P:2} is not true.

Let $G$ denote the general linear group of  invertible $n\times n$ matrices. 
Let $B$ denote the Borel subgroup of invertible upper triangular matrices in $G$, 
and let $B^-$ denote the opposite Borel subgroup consisting of invertible lower 
triangular matrices in $G$.
Let $U$ denote the maximal unipotent subgroup of $B$. 
Since $UB^-$ is open in $G$, we see that $U$ is a spherical subgroup in $G$.
However, $(U,U)$ is not a spherical pair in $G$ since $2\dim U = n(n-1) < \dim G= n^2$. 
\end{Example}

Let $(G_1,G_2)$ be a spherical pair in $G$. 
Composing the action with the inverse automorphism, $\iota: x\mapsto x^{-1}$ ($x\in G$), we see that $(G_2, G_1)$ is a spherical pair, also.

\begin{Lemma}\label{L:foreveryg}
Let $(G_1,G_2)$ be a pair of closed subgroups from a connected reductive group $G$.
Then the following statements are equivalent:
\begin{enumerate}
\item[(1)] $(G_1,G_2)$ is a spherical pair in $G$; 
\item[(2)] for every $g\in G$, $(G_1,gG_2g^{-1})$ is a spherical pair in $G$; 
\item[(3)] for every $g\in G$, $(gG_1g^{-1}, G_2)$ is a spherical pair in $G$.
\end{enumerate}
\end{Lemma}

\begin{proof}

The equivalence of (2) and (3) is obvious. 
Let us show the implication (1) $\Rightarrow$ (2). 
Let $B_1\times B_2$ be a Borel subgroup in $G_1\times G_2$ such that there exists $h\in G$ with $B_1hB_2$ open in $G$.  
Let $g$ be an element from $G$, and let $^g\!B_2$ denote the conjugate subgroup $gB_2g^{-1}$. 
Clearly, $^g\! B_2$ is a Borel subgroup of $g G_2g^{-1}$, and furthermore, $B_1 h B_2 = B_1 h g^{-1}  {}^g\!B_2 g$. 
But $B_1 h g^{-1}  {}^g\!B_2 g$ is open in $G$ if and only if its translation by any element of $G$ is open.
Therefore, $B_1 h g^{-1} {}^g\! B_2$ is open in $G$. 
In other words, the Borel subgroup $B_1\times {^g\!B_2}$ of $G_1\times gG_2g^{-1}$ has an open orbit in $G$. 
Finally, to see the truth of the implication (2) $\Rightarrow$ (1), let $e$ denote the identity element in $G$.
Then $(G_1,eG_2e^{-1})=(G_1,G_2)$ is a spherical pair. 
This finishes the proof.
\end{proof}

\begin{Proposition}\label{P:1}
Let $(G_1,G_2)$ be a pair of closed subgroups from a connected reductive group $G$.
We assume that one of the subgroups, say $G_2$, contains a Borel subgroup of $G$.
Then the following statements are equivalent:
\begin{enumerate}
\item[(1)] $(G_1,G_2)$ is a spherical pair;
\item[(2)] $G_1$ is a spherical subgroup;
\item[(3)] $G_1$ and $G_2$ are spherical subgroups. 
\end{enumerate}
\end{Proposition}

\begin{proof}
Since $G_2$ contains a Borel subgroup, it is automatically a spherical subgroup in $G$. 
Therefore, it suffices to show the equivalence of (1) and (2). 

The implication (1) $\Rightarrow$ (2) follows from Proposition~\ref{P:2}.
We proceed to show (2) $\Rightarrow$ (1).

Let $B$ denote the Borel subgroup of $G$ which is contained in $G_2$. 
Then $B$ is a Borel subgroup of $G_2$ as well. 
By~\cite[Theorem 22.6]{Grosshans}, since $G_1$ is spherical, 
we know that any Borel subgroup of $G$ has a finite number of orbits in $G/ G_1$.
Equivalently, for any Borel subgroup $B$ of $G$, $BG_1$, hence $G_1B$, is open in $G$.

Let $\mc{R}_u(G_1)$ denote the unipotent radical of $G_1$. 
Since $G_1$ is connected, $G_1=\mc{R}_u (G_1) \rtimes G_1'$,
where $G_1'$ is the reductive quotient of $G_1$. 
Note that $\mc{R}_u(G_1)$ is contained in the unipotent radical of a Borel subgroup $B_1$ in $G_1$. 
We have the Bruhat-Chevalley decomposition, $G_1 = B_1 W_r B_1$, where $W_r$ is the Weyl group of the reductive quotient $G_1'$. 
It follows from Lemma~\ref{L:foreveryg} that replacing $G_2$ with a conjugate subgroup does not cause any harm. 
Therefore, we assume that $B_1$ is contained in the Borel subgroup $B$ of $G_2$.

Let $w_0\in W_r$ denote the Weyl group element such that $B_1w_0B_1$ is open in $G_1$. 
Then the $B_1\times B$-orbit of $w_0$ in $G$ is equal to 
\begin{align}\label{A:denseorb}
B_1 w_0 B = B_1 w_0 B_1 B.
\end{align}
But since $G_1B$ is open in $G$, and since $B_1 w_0 B_1$ is open in $G_1$, we see that the orbit (\ref{A:denseorb}) is open in $G$. 
Thus, $(G_1,G_2)$ is a spherical pair in $G$.

\end{proof}

\begin{Lemma}\label{L:goodresult}
Let $G$ be a connected reductive group,  
and let $G_1$ and $G_2$ be two closed connected subgroups such that $G=G_1G_2$. 
If $(G_1,G_2)$ is a spherical pair, then the homogeneous space $G_i / G_1\cap G_2$, for 
$i\in \{1,2\}$, is a spherical $G_i$-variety.
\end{Lemma}

\begin{proof}

Let $H$ denote the subgroup $G_1\cap G_2$ of $G$. 
Since $G_1G_2$ is a decomposition of $G$, 
we have $\text{Stab}_{G_1\times G_2} (e) =  \{ (a,a^{-1}):\ a\in H\}$. 
There is a surjective $G_1\times G_2$-equivariant morphism, 
\begin{align*}
f:   (G_1\times G_2)/ \text{Stab}_{G_1\times G_2}(e) & \to G_1/H \times G_2/H \\
(g_1,g_2)  \text{Stab}_{G_1\times G_2}(e) &\mapsto (g_1H, g_2H)
\end{align*}
Therefore, if the homogeneous space 
at the source of $f$, that is $G$, is a spherical $G_1\times G_2$-variety, then 
so is the target $G_1/H \times G_2/H$. But $G_1/H \times G_2/H$ is $G_1\times G_2$-spherical
if and only if $G_1/H$ is $G_1$-spherical and $G_2/H$ is $G_2$-spherical.

\end{proof}

Before stating our next result, we look at Example~\ref{E:goodexample} once more.

\begin{Example}\label{E:goodexample2}
Let $G_1$ denote $U$, and let $G_2$ denote $B^-$. 
Then $(G_1,G_2)$ is a spherical pair in $G$. 
Also, it is easy to see that $G_1G_2 \neq G$. These observations show that 
there is a rather subtle relationship between decompositions
and spherical pairs.
\end{Example}

\begin{Definition}\label{D:Half}
We call a pair of closed subgroups $(G_1,G_2)$ {\em nontrivial} if $G_1\neq G$ and $G_2\neq G$; 
otherwise, we call $(G_1,G_2)$ a {\em trivial pair}.
We call $(G_1,G_2)$ a {\em reductive pair} if both of the groups $G_1$ and $G_2$ are reductive groups. 
By the same token, we call a nontrivial pair $(G_1,G_2)$ {\em half-reductive} if exactly one of the subgroups
$G_1$ and $G_2$ is a reductive group. 
\end{Definition}

\begin{Proposition}\label{P:noreductivepairs}
There are no nontrivial reductive spherical pairs.
\end{Proposition}
\begin{proof}
Towards a contradiction, let $(G_1,G_2)$ be a nontrivial spherical pair,
where $G_1$ and $G_2$ are reductive subgroups in $G$. 
Let $B_1\times B_2$ be a Borel subgroup of $G_1\times G_2$ such that $B_1B_2$ is open in $G$. 
Since $(G_1,G_2)$ is a nontrivial pair, and since $G_1$ and $G_2$ are reductive subgroups, 
neither $B_1$ nor $B_2$ is a Borel subgroup in $G$. 
Now, if it is necessary, then replacing $G_2$ by a conjugate subgroup in $G$, we assume that $B_1$ 
is contained in a Borel subgroup $B$ of $G$ and that 
$B_2$ is contained in the opposite Borel subgroup $B^-$. 
Let $U$ (respectively, $U^-$) denote the unipotent radical of $B$ (respectively of $B^-$).
Since $B_1$ (respectively, $B_2$) is properly contained in $B$
(respectively in $B^-$), its unipotent radical is properly contained in $U$ (respectively in $U^-$).
In other words, $\dim B_1 < \dim B$ and $\dim B_2 < \dim B^- = \dim B$. 
Since $BB^- \cong U \times T \times U^-$ in $G$, we see that $\dim B_1 B_2 <  \dim G+1$. 
This contradiction finishes the proof. 
\end{proof}

The proof of Proposition~\ref{P:noreductivepairs} shows also that there are no reductive complexity one pairs. 

\begin{Corollary}\label{C:noreductivepairs}
There are no nontrivial reductive complexity one pairs.
\end{Corollary}

Next, we will show that there are no half-reductive complexity zero pairs. 
For this purpose, we have a preliminary result involving horospherical varieties: 
A closed subgroup $H\subseteq G$ is called {\em horospherical} 
if $H$ contains a maximal unipotent subgroup of $G$. 
\begin{Lemma}\label{L:Horosphericaliff}
Let $G$ be a connected reductive group, and let $(G_1,G_2)$ be a spherical pair in $G$. 
Then $G_1$ is a horospherical subgroup if and only if $G_2$ is a horospherical subgroup. 
\end{Lemma}

\begin{proof}
We argue as in the proof of Proposition~\ref{P:noreductivepairs} by letting $B_1\times B_2$ 
be a Borel subgroup of $G_1\times G_2$ such that $B_1 \subseteq B^-$ and $B_2\subseteq B$.
Since $B_1B_2 \subseteq B^-B \cong U^- \times T \times U$, we see that, if $U^-\subset B_1$
but $U \nsubseteq B_2$, then $\dim B_1B_2 < \dim B^-B=\dim G$, which contradicts with our initial assumption
that $(G_1,G_2)$ is a spherical pair. 
\end{proof}

\begin{Proposition}\label{P:nohalfreductive}
There are no half-reductive complexity zero pairs.
\end{Proposition}

\begin{proof}

Let $G$ be a connected reductive group, and let $(G_1,G_2)$ be a half-reductive pair, 
where $G_1$ is the reductive subgroup. 
Towards a contradiction we assume that $(G_1,G_2)$ is a complexity zero pair. 

First, we will show that if $(G_1,G_2)$ is a complexity zero pair in $G$, then $G_2$ is a horospherical subgroup.
To this end, we use, once again, the idea in the proof of Proposition~\ref{P:noreductivepairs} 
by letting $B_1\times B_2$ be a Borel subgroup of 
$G_1\times G_2$ such that $B_1 \subseteq B^-$ and $B_2\subseteq B$.
Since $G_1$ is reductive, $B_1$ is properly contained in $B_1$. 
Therefore, if $B_2$ does not contain $U$, then $\dim B_1B_2 < \dim B^-U = \dim G$, which 
contradicts the assumption that $(G_1,G_2)$ is a complexity zero pair. 
This contradiction shows that if $(G_1,G_2)$ is a half-reductive complexity zero pair with $G_1$
a reductive group, then $G_2$ has to be a horospherical subgroup. 

Now, by Lemma~\ref{L:Horosphericaliff}, we know that $G_2$ is horospherical if and only if 
$G_1$ is horospherical, so $U\subset G_1$. But since $G_1$ is reductive, this implies that $G_1=G$.
This is a contradiction, hence, the proof is complete.

\end{proof}

\section{Observations About Complexity One Pairs}\label{S:ComplexityOne}

The proof of the following fact is similar to the proof of the corresponding 
statement for spherical pairs, so we omit it. 

\begin{Lemma}\label{L:foreveryg}
Let $(G_1,G_2)$ be a complexity $c$ pair in a connected reductive group $G$.
Then the following statements are equivalent:
\begin{enumerate}
\item[(1)] $(G_2,G_1)$ is a complexity $c$ pair in $G$;
\item[(2)] for every $g\in G$, $(G_1,gG_2g^{-1})$ is a complexity $c$ pair in $G$;
\item[(3)] for every $g\in G$, $(gG_1g^{-1}, G_2)$ is a complexity $c$ pair in $G$.
\end{enumerate}
\end{Lemma}

\begin{Lemma}\label{L:interchange}
Let $G$ be a connected algebraic group, and let $B$ be a Borel subgroup in $G$. 
Let $H$ be a closed subgroup of $G$. If $c_{H}(G/B) =i$, where $i\in \{0,1\}$, then $c_G(G/H) \leq i$.  
\end{Lemma}

\begin{proof}
Let $B_1$ be a Borel subgroup of $H$ with an orbit of codimension one in $G/B$. 
Then 
$$
i = \min_{xB \in G/B} \text{codim}_{G/B} (B_1 \cdot xB), \text{ or, equivalently, } i= \min_{x\in G} (\dim G - \dim B_1 x B). 
$$
But $\dim H x B \geq B_1 x B$ for all $x\in G$. Therefore, 
$$
i\geq \min_{x\in G} (\dim G - \dim H x B) =  \min_{xH \in G/H} \text{codim}_{G/H} (B \cdot xH).
$$
This finishes the proof.
\end{proof}

\begin{Proposition}\label{P:sumleq1}
Let $G$ be a connected reductive algebraic group.  
If $(G_1,G_2)$ is a complexity one pair in $G$, then 
$c_G(G/G_1) + c_G(G/G_2) \leq 1$.
\end{Proposition}

\begin{proof}
Let $(G_1,G_2)$ be a complexity one pair in $G$, and 
let $X$ denote the wonderful compactification of $G/Z(G)$.
As in the proof of Proposition~\ref{P:2}, let $r$ denote the semisimple rank of $G$,
so that the $G\times G$-orbit closures in $X$ are indexed by the subsets of $[r]$.

Clearly, $X$ is a $G_1\times G_2$-variety. 
Since the action of $G_1\times G_2$ on $G$ is of complexity one, 
and $G$ is open in $X$, we see that the action $G_1\times G_2:X$ is of complexity one, as well. 
In particular, since every $G\times G$-orbit closure $X_I$ in $X$ is $G_1\times G_2$-stable, 
by Lemma~\ref{L:Knop}, we know that $c_{G_1\times G_2} (X_I) \leq c_{G_1\times G_2} (G) = 1$.
Recall that the closed $G\times G$-orbit in $X$ is isomorphic to the double-flag variety,
\hbox{$X_{[r]} \cong G/B^- \times G/B$}, 
and the action $G\times G: X_{[r]}$ is given by (\ref{A:GGaction}).

Let $B_1\times B_2$ be a Borel subgroup of $G_1\times G_2$. 
Then 
\begin{align*}
1\geq c_{G_1\times G_2 }(X_{[r]}) 
&=\min_{(x_1B^-,x_2B)\in X_{[r]}} \text{codim}_{X_{[r]}}  (B_1\times B_2 \cdot (x_1B^-,x_2B))  \\
&=\min_{x_1B^- \in G/B^-}  \text{codim}_{G/B^-} B_1\cdot x_1B^- +  \min_{x_2B \in G/B}  \text{codim}_{G/B} B_2\cdot x_2B\\
&= c_{G_1} (G/B^-) + c_{G_2} (G/B).
\end{align*}
It follows from this inequality and Lemma~\ref{L:interchange} that 
\hbox{$c_{G}(G/G_1) + c_{G}(G/G_2) \leq 1$}.
This finishes the proof. 
\end{proof}

\section{Diagonal Actions}\label{S:Diagonal}

We begin with reviewing some results about decompositions and double cosets of reductive groups. 

\begin{Theorem}[Theorem 2.1~\cite{Onishchik2}]
Let $K$ be a compact connected Lie group, $K'$ and $K''$ closed connected subgroups
where $K=K'K''$. Then $K^\C = K'^{\C} K''^{\C}$. 
Conversely, let $G$ be a connected reductive algebraic group over $\C$
and let $G=G'G''$, where $G'$ and $G''$ are connected complex Lie subgroups 
and $G' \cap G''$ has a finite number of connected components. 
Let $K,K',K''$ be maximal compact subgroups in $G,G',G''$, where $K\supset K',K''$.
Then $K=K'K''$. If $G_r'$ and $G_r''$ are maximal reductive algebraic subgroups 
in $G'$ and $G''$, then $G=G_r'G_r''$.
\end{Theorem}

Note that, in Onishchik's theorem, the finiteness of the connected components of $G'\cap G''$
is automatically satisfied if we assume that $G$ is a connected reductive group, and that 
$G'$ and $G''$ are algebraic subgroups. 
In his earlier work~\cite{Onishchik1}, Onishchik gave a complete list of the decompositions 
into compact Lie groups of simple compact Lie groups, 
or equivalently, the decompositions into compact Lie subalgebras 
of simple compact Lie algebras.
For the benefit of the reader, we listed the complexifications of these decompositions 
in Table~\ref{T:3} in the appendix.

\begin{Theorem}[Luna~\cite{Luna72}]\label{T:Luna}
Let $G$ be a connected reductive algebraic group, and 
let $(G_1,G_2)$ be a pair of reductive subgroups from $G$. 
Then the union of closed $(G_1,G_2)$-double cosets 
in $G$ contains an open dense subset of $G$.
\end{Theorem}
Luna's theorem has a useful consequence.
\begin{Corollary}\label{C:Luna}
Let $G,G_1$, and $G_2$ be as in Theorem~\ref{T:Luna}. 
If the number of $(G_1,G_2)$-double cosets in $G$ is finite, 
then $G_1G_2 = G$. 
\end{Corollary}

For type A, the complete classification of spherical actions of reductive subgroups
on flag varieties is found by Avdeev and Pethukov in~\cite{AP}. 
According to Definition~\ref{D:Half} and Proposition~\ref{P:1}, this progress is equivalent to classification of 
half-reductive spherical pairs of the form $(P,H)$ in $G$, where $P$ is a parabolic subgroup
and $H$ is a reductive subgroup. 

In this section, we will consider a closely related problem; 
we will classify the diagonal actions $G: G/P\times G/H$, where $P$ is a parabolic subgroup, 
and $H$ is a reductive spherical subgroup.
Note that the classification of diagonal spherical actions 
on products of flag varieties is known, see the papers~\cite{Littelmann,
MWZ1,MWZ2,Stembridge}, as well as~\cite{Ponomareva}.

\begin{Theorem}\label{T:SecondClassification}
Let $G$ be a simple group, and let $H$ be a spherical reductive subgroup of $G$. 
Let $P$ be a parabolic subgroup of $G$, and let $K$ be a Levi subgroup in $P$.
Then the diagonal action of $G$ on $G/P\times G/H$ is spherical if and only if 
$(H,K)$ is one of the pairs in Table~\ref{T:1} corresponding to the rows 2,4,5, and 7.
The first three of these rows, that are 2,4, and 5, correspond to the symmetric subgroups $H$ in $G$ 
such that the diagonal action $G: G/P\times G/H$ is spherical.
\end{Theorem}

\begin{proof}
By Lemma~\ref{L:itsolvesit}, we know that $G/P\times G/H$ is spherical if and only if $G/H$ is a spherical $K$-variety. 
In particular, $K$ has only finitely many orbits in $G/H$. Equivalently, there are only finitely many $(H,K)$-double cosets in $G$. 
Since both $K$ and $H$ are reductive groups, it follows from Corollary~\ref{C:Luna} that $G=HK$. 
The list of all such pairs $(H,K)$, where $H$ and $K$ are spherical subgroups in $G$, and $G=HK$ is 
easy to find by inspection of the Tables~\ref{T:3} and~\ref{T:4}. They are given by the rows of Table~\ref{T:1}. 
\begin{table}[htp]
\begin{center}
\begin{adjustbox}{max width = 4.5in}
\begin{tabular}{l@{\hskip .35in}l@{\hskip .35in}l@{\hskip .35in}l}
 no. & $G$ 		&  $H$     &    $K$   \\     
 \hline \\
1.&  $\text{SL}(2n)$,  $n>1$ &   $\text{Sp}(2n)$    & $\text{SL}(2n-1)$  \\ 
2.&  $\text{SL}(2n)$, $n>1$ &   $\text{Sp}(2n)$    & $S( \text{GL}(1)\times \text{GL}(2n-1))$   \\ \\

3. &  \begin{tabular}{l}$\text{SO}(2n+2)$, $n>2$\\ and $n$ is even\end{tabular} &   
$\text{SO}(2n+1)$    & $\text{SL}(n+1)$       \\  
4. &  \begin{tabular}{l}$\text{SO}(2n+2)$, $n>2$\\ and $n$ is even\end{tabular} &   
$\text{SO}(2n+1)$    & $\mathbf{G}_m\cdot  \text{SL}(n+1)$     \\   \\

5. & $\text{SO}(8)$ & $\text{SO}(7)$ & $\mathbf{G}_m\cdot  \text{SO}(6)$  \\    
6. & $\text{SO}(8)$ & $\text{SO}(7)_+$ & $\text{SO}(7)_-$  \\    \\

7. & $\text{SO}(7)$ & $\text{G}_2$ & $\mathbf{G}_m\cdot  \text{SO}(5)$   \\     
8. & $\text{SO}(7)$ & $\text{G}_2$ & $\text{SO}(6)$ 
\end{tabular}
\end{adjustbox}
\caption{Reductive pairs of spherical subgroups decomposing simple algebraic groups.}
\label{T:1}
\end{center}
\end{table}
By Proposition~\ref{P:2}, we have an additional restriction; the subgroup $K$ must be a spherical Levi subgroup in $G$. 
Such Levi subgroups are first classified by Brion in characteristic zero, see~\cite[Proposition 1.5]{Brion87}. 
By comparing with Brion's list, we see that the rows with no. 2,4,5, and 7 correspond to the triplets 
$(G,H,K)$ such that $H$ and $K$ are reductive spherical subgroups of $G$, and $K$ is a Levi subgroup of some parabolic subgroup of $G$. 

Finally, our last assertion follows from the well known classification of symmetric subgroups of simple groups, 
see~\cite[Table 26.3]{Timashev}.
\end{proof}

\begin{Remark}
Brion's classification of spherical Levi subgroups of simple groups is 
extended to positive characteristic by Brundan in~\cite[Theorem 4.1]{Brundan}.
\end{Remark}

\section{Classification of Half-reductive Complexity One Pairs}\label{S:Half-reductive}

In this section, we assume that $G$ is simple
and that $(G_1,G_2)$ is a nontrivial half-reductive complexity one pair with reductive $G_1$. 

Let $B_1\times B_2$ be a Borel subgroup in $G_1\times G_2$ and let $B$ be a Borel subgroup in $G$
such that $B_1\subset B$ and $B_2 \subset B^-$, where $B^-$ is the opposite Borel subgroup 
corresponding to $B$. Since $(G_1,G_2)$ is a complexity one pair, we know that 
$$
\dim B_1B_2 +1= \dim G  = \dim BB^- = \dim U + \dim T + \dim U^-,
$$
where $T=B\cap B^-$ is a maximal torus, and $U$ and $U^-$ are the unipotent radicals of $B$ and $B^-$, 
respectively. 
Since $G_1$ is a reductive subgroup of $G$, we know that $B_1$ is properly contained in $B$,
therefore, the unipotent radical $U_1$ of $B_1$ is properly contained in $U$. 
It follows from dimension considerations that $U_2 = U$, and that $\dim U_1 = \dim U -1$.
We note here that, by Proposition~\ref{P:sumleq1}, the complexity of $G_1$ is either 0 or 1. 
Furthermore, $\text{rk}(G)-\text{rk}(G_1) = 1$. 
These observations will give us the complete classification of 
nontrivial, half-reductive, complexity one pairs in simple algebraic groups. 

\begin{Theorem}\label{T:FirstClassification}
Let $G$ be a simple algebraic group, and let $(G_1,G_2)$ be a half-reductive complexity one pair
with reductive $G_1$. Then $G_2$ is a horospherical subgroup and $(G,G_1)$ is 
one of the following pairs: 
\begin{enumerate}
\item[(1)] $(\text{SO}(4),\text{GL}(2))$,  
\item[(2)] $(\text{SO}(4),\text{SO}(3))$, 
\item[(3)] $(\text{SO}(3),\text{SO}(2))$,
\item[(4)] $(\text{SL}(2),\{e\})$. 
\end{enumerate}
Furthermore, in the cases of (1) and (3),  the horospherical subgroup $G_2 \subset G$ can be any horospherical subgroup;
in the case of (2), $G_2$ is isomorphic to one of $\text{SL}_2 \times \mathbf{G}_a$, or $\text{SL}_2 \times B_2$, where 
$B_2$ is the Borel subgroup of upper triangular matrices in $\text{SL}(2)$, and in the case of (4), we have $G_2=B_2$.
\end{Theorem}

\begin{proof}
We already observed in the previous paragraph that $G_2$ is horospherical. 
Let $Q$ be a parabolic subgroup in $G$ admitting a regular embedding $G_1 \hookrightarrow Q$. 
By Proposition~\ref{P:sumleq1} we know that $c_G(G_1) \leq 1$. The list of spherical and complexity one reductive
subgroups is given in Table~\ref{T:4}. It is easy to compute the dimensions of the maximal unipotent subgroups
of these reductive subgroups. It turns out that there are only three pairs $(G,G_1)$ 
such that $\dim U_1 = \dim U -1$, where $U_1\subset U$ are maximal unipotent subgroups of $G_1$ and $G$, respectively. 
After using the coincidences between simple groups in small ranks, we see that these pairs are as in the
statement of the theorem. This is a rather tedious computation, so, we omit the details. 

For the last claim, since we already know that $G_2$ contains a maximal unipotent subgroup $U\subset G$, 
we only need to understand the intersections $T\cap G_2$, where $T\subset G$ is a maximal torus which normalizes $U$. 
In the cases of (1) and (3), we see that $G_1$ contains $T$, therefore, $G_2$ can be any horospherical subgroup of $G$. 
In the case of (2) we use the identification $\text{SO}(4) \cong \text{SL}(2)\times \text{SL}(2)$, and we observe that 
$\mathbf{G}_m\times \{e\} \subset G_2$. Finally, in the case of (4), we have $\mathbf{G}_m \subset G_2$. 

\end{proof}

\begin{Remark}
By using the idea of the proof of Theorem~\ref{T:FirstClassification} one can obtain the classification of triplets $(G,G_1,G_2)$,
where $G$ is a simple algebraic group, $G_1$ is a reductive subgroup, $G_2$ is a horospherical subgroup of $G$,
and $(G_1,G_2)$ is a complexity two pair in $G$. 
\end{Remark}

\begin{Remark}
Let $G$ be a connected reductive group, 
let $B$ be a Borel subgroup in $G$, and let $U$ denote the maximal unipotent subgroup of $B$. 
Let $S$ be horospherical subgroup with $U\subseteq S$, and let $P$ denote the normalizer of $S$ in $G$.
Then $P$ is a parabolic subgroup. Let $P= \mathcal{R}_u(P)\rtimes L$ denote Levi decomposition of $P$,
where $\mathcal{R}_u(P)$ is the unipotent radical of $P$ and $L$ is a Levi subgroup. Then 
$S= \mathcal{R}_u(P)\rtimes L_0$, where $L_0$ is a Levi subgroup such that 
$L' \subseteq L_0 \subseteq L$, see~\cite[Section 7]{Timashev}. 
Here, $L'$ denotes the commutator subgroup of $L$. 
\end{Remark}

\section{Appendix}\label{S:Appendix}

The first table gives the complexifications of the data of decompositions of compact Lie groups into compact Lie groups.
This table is due to Onishchik.  
In the second table, we have two columns. 
In the first column we have Kr\"amer's list of reductive subgroups in simple groups in characteristic zero.
This list is shown to be valid in arbitrary characteristic by Brundan in~\cite{Brundan}. 
Recently, it is shown by Knop and R\"ohrle~\cite{KnopRohrle} that there is one additional subgroup 
when characteristic is 2. In characteristic zero, there is a classification of reductive spherical subgroups in reductive groups.
This achievement is due to Brion~\cite{Brion87} and Mikityuk~\cite{Mikityuk}.
\begin{table}[htp]
\begin{center}
\begin{adjustbox}{max width = 5.5in}
\begin{tabular}{l@{\hskip .35in}l@{\hskip .35in}l@{\hskip .35in}l}
  $K^\C$  &  $K'^\C$     &    $K''^\C$ & $K'^\C \cap K''^\C$ \\     
 \hline \\
 $\text{SL}(2n)$, $n>1$ &   $\text{Sp} (2n)$    & $\text{SL}(2n-1)$  & $\text{Sp} (2n-2)$ \\
 $\text{SL}(2n)$,  $n>1$ &   $\text{Sp} (2n)$  & $S( \text{GL}(1)\times \text{GL}(2n-1))$  & $\mathbf{G}_m\cdot  \text{Sp} (2n-2)$     \\ \\

 $\text{SO}(2n+2)$,  $n>2$ &   $\text{SO}(2n+1)$    & $\text{SL}(n+1)$     &   $\text{SL}(n)$     \\   

 $\text{SO}(2n+2)$,  $n>2$ &   $\text{SO}(2n+1)$    &   $\mathbf{G}_m\cdot  \text{SL}(n+1)$      &   $\mathbf{G}_m\cdot  \text{SL}(n)$     \\   \\

 $\text{SO}(4n)$,  $n>1$ &   $\text{SO}(4n-1)$    & $\text{Sp} (2n)$  & $\text{Sp} (2n-2)$          \\   
 $\text{SO}(4n)$,  $n>1$ &    $\text{SO}(4n-1)$    &   $\mathbf{G}_m\cdot  \text{Sp} (2n)$ & $\mathbf{G}_m\cdot  \text{Sp} (2n-2)$    \\   
 $\text{SO}(4n)$,  $n>1$ &   $\text{SO}(4n-1)$  & $\text{SL}(2)\times \text{Sp} (2n)$  & $\text{SL}(2)\times \text{Sp} (2n-2)$       \\   \\
 
 $\text{SO}(16)$ & $\text{SO}(15)$ & $\text{SO}(9)$	& $\text{SO}(7)$\\     \\
 $\text{SO}(8)$ & $\text{SO}(7)$ & $\text{SO}(5)$ & $\text{SL}(2)$\\    
 $\text{SO}(8)$ & $\text{SO}(7)$ & $\mathbf{G}_m\cdot  \text{SO}(5)$ & $\mathbf{G}_m\cdot  \text{SL}(2)$\\    
 $\text{SO}(8)$ & $\text{SO}(7)$ & $\text{SL}(2)\times \text{SO}(5)$ & $\text{SL}(2)\times \text{SL}(2)$\\    
 $\text{SO}(8)$ & $\text{SO}(7)$ & $\text{SO}(6)$ & $\text{SL}(3)$\\    
 $\text{SO}(8)$ & $\text{SO}(7)$ & $\mathbf{G}_m\cdot  \text{SO}(6)$ & $\mathbf{G}_m\cdot  \text{SL}(3)$\\    
 $\text{SO}(8)$ & $\text{SO}(7)$ & $\text{SO}(7)$ & $\text{G}_2$\\  \\
 
$\text{SO}(7)$ & $\text{G}_2$ & $\text{SO}(5)$  &   $\text{SL}(2)$\\       
 $\text{SO}(7)$ & $\text{G}_2$ & $\mathbf{G}_m\cdot  \text{SO}(5)$	& $\mathbf{G}_m\cdot  \text{SL}(2)$	\\     
 $\text{SO}(7)$ & $\text{G}_2$ & $\text{SO}(6)$  &	$\text{SL}(3)$ \\     
\end{tabular}
\end{adjustbox}
\caption{The decompositions of simple algebraic groups into reductive algebraic subgroups.}
\label{T:3}
\end{center}
\end{table}

\begin{table}[htp]
\begin{center}
\begin{adjustbox}{max width = 6.5in}
\begin{tabular}{l@{\hskip .5in}l}
  
 \begin{tabular}{c}$(G,H)$, where $G$ is simple and $H$ is a\\ complexity zero reductive subgroup \end{tabular} & 
 \begin{tabular}{c}$(G,H)$, where $G$ is simple and $H$ is a\\ complexity one reductive subgroup \end{tabular}    \\    
 \hline
 \\
$(\text{SL}(n), \text{SO}(n))$, $n\geq 2$  		&	    $(\text{SL}(2n), \text{SL}(n)\times \text{SL}(n))$  \\
$(\text{SL}(n+m), S(\text{GL}(m)\times \text{GL}(n)) )$, $m\geq n\geq 1$  &   $(\text{SL}(n),\mathbf{G}_m \times \text{SL}(n-2))$, $n \geq 5$ \\    
$(\text{SL}(n+m), \text{SL}(m)\times \text{SL}(n))$, $m> n\geq 1$   		& $(\text{SL}(n), \mathbf{G}_m^2 \times \text{SL}(n-2))$, $n \geq 3$    \\    
$(\text{SL}(2n+1), \mathbf{G}_m\cdot \text{Sp}(2n))$, $n\geq 1$	  &   $(\text{SL}(6), \mathbf{G}_m \times \text{SL}(2) \times \text{Sp}(4))$ \\
$(\text{SL}(2n+1), \text{Sp}(2n))$, $n\geq 1$ 				 &  \\
$(\text{SL}(2n), \text{Sp}(2n))$, $n\geq 2$ 				 &   \\ \\

$(\text{Sp}(2n), \text{GL}(n))$, $n\geq 1$	  &    $(\text{Sp}(2n), \text{Sp}(2n-2))$ \\
$(\text{Sp}(2n), \mathbf{G}_m\times \text{Sp}(2n-2))$, $n\geq 2$	  &  $(\text{Sp}(2n), \text{Sp}(2n-4) \times \text{SL}(2) \times\text{SL}(2) )$, $n\geq 3$  \\
$(\text{Sp}(n+m), \text{Sp}(n)\times \text{Sp}(m))$, $n,m\geq 2$, even 	&  $(\text{Sp}(2n), \text{SL}(n))$, $n\geq 2^*$  \\ \\

$(\text{SO}(2n), \text{SL}(n))$, $n\geq 3$, odd	  & 	$(\text{SO}(n), \text{SO}(n-2))$, $n\geq 4$  \\
$(\text{SO}(2n), \text{GL}(n))$, $n\geq 2$	  & 	$(\text{SO}(2n+1), \text{SL}(n))$  \\
$(\text{SO}(2n+1), \text{GL}(n))$, $n\geq 2$	  & 	$(\text{SO}(4n), \text{SL}(2n))$, $n\geq 2$  \\
$(\text{SO}(n+m), \text{SO}(m)\times \text{SO}(n))$, $m\geq n\geq 1$ 	& 	$(\text{SO}(11), \text{SO}(3)\times \text{Spin}(7))$ \\
$(\text{SO}(7), \text{G}_2)$	  & 	$(\text{SO}(10), \text{Spin}(7))$	 \\
$(\text{SO}(8), \text{G}_2)$	  & 	$(\text{SO}(9), \text{G}_2\times \text{SO}(2))$ \\
$(\text{SO}(9), \text{Spin}(7))$	  & 	\\
$(\text{SO}(10), \text{SO}(2)\times \text{Spin}(7))$	  & 	  \\ \\

$(\text{G}_2,\text{A}_2)$	  & 	\\
$(\text{G}_2,\text{A}_1\times \widetilde{\text{A}}_1)$	  & 	\\
$(\text{F}_4,\text{B}_4)$	  &  $(\text{F}_4,\text{D}_4)$	\\
$(\text{F}_4,\text{C}_3\times \text{A}_1)$	  & 	\\
$(\text{E}_6,\text{C}_4)$	  & 	$(\text{E}_6, \mathbf{G}_m\times \text{B}_4)$ \\
$(\text{E}_6,\text{F}_4)$	  & 	\\
$(\text{E}_6,\text{D}_5)$	  & 	\\ 
$(\text{E}_6,\mathbf{G}_m\cdot \text{D}_5)$	  & 	\\ 
$(\text{E}_6,\text{A}_5\times \text{A}_1)$	 & 	\\ 

$(\text{E}_7,\mathbf{G}_m\cdot \text{E}_6)$	  & $(\text{E}_7,\text{E}_6)$	\\ 
$(\text{E}_7,\text{A}_7)$	  & 	\\ 
$(\text{E}_7,\text{D}_6\times \text{A}_1)$	 & 	\\ 

$(\text{E}_8,\text{D}_8)$	  & 	\\
$(\text{E}_8,\text{E}_7\times \text{A}_1)$	  & 	\\

\end{tabular}
\end{adjustbox}
\caption{Spherical and complexity one reductive subgroups in simple algebraic groups.}
\label{T:4}
\end{center}
\end{table}

\newpage

\bibliography{References}
\bibliographystyle{plain}

\end{document}